\documentclass[12pt]{amsart}
\usepackage{amssymb}
\usepackage{hyperref}
\usepackage[dvips]{graphics}

\allowdisplaybreaks[1]

\newtheorem{theorem}{Theorem}[section]
\newtheorem{definition}[theorem]{Definition}

\newtheorem{proposition}[theorem]{Proposition}

\begin{document}

\title{Spectral analysis of the free orthogonal matrix}

\author{T. Banica}
\address{T.B.: Department of Mathematics, Toulouse 3 University, 118 route de Narbonne, 31062 Toulouse, France. {\tt banica\,@\,math.ups-tlse.fr}}

\author{B. Collins}
\address{B.C.: Department of Mathematics, Ottawa University and Lyon 1 University, 43 bd du 11 novembre 1918, 69622 Villeurbanne, France. {\tt collins\,@\,math.univ-lyon1.fr}}

\author{P. Zinn-Justin}
\address{P.Z.-J.: Laboratoire de Physique Th\'eorique des Hautes Energies (CNRS, UMR 7589), Univ Pierre et Marie Curie-Paris6, 75252 Paris Cedex, France. {\tt pzinn\,@\,lpthe.jussieu.fr}}


\begin{abstract}
We compute the spectral measure of the standard generators $u_{ij}$ of the Wang algebra $A_o(n)$. We show in particular that this measure has support $[-2/\sqrt{n+2},2/\sqrt{n+2}]$, and that it has no atoms. The computation is done by using various techniques, involving the general Wang algebra $A_o(F)$, a representation of $SU^q_2$ due to Woronowicz, and several calculations with orthogonal polynomials.
\end{abstract}

\maketitle

\section*{Introduction}

The compact quantum groups were introduced by Woronowicz in \cite{wo1}, \cite{wo2}, \cite{wo3}. They  generalize the compact groups of unitary matrices. They have a number of 
properties similar to compact matrix groups, such as
 a Peter--Weyl type theory for their finite dimensional representations, a Tannaka--Krein type duality, and the existence of a unique left and right invariant Haar state.

The first new examples were given by Woronowicz as one-parameter deformations of the classical Lie groups. These provide compact versions for the algebras $U_q(g)$ constructed by Drinfeld \cite{dri} and Jimbo \cite{jim}, in the case $q\in\mathbb R$.
The first examples with new fusion rules and dimensions were constructed by Wang in the mid-nineties, in two influential papers \cite{wa1}, \cite{wa2}. These quantum groups appear as free versions of the orthogonal, unitary and symmetric groups, $O_n,U_n,S_n$. 

A key point in the understanding of these quantum groups is the explicit computation of their Haar measures. The computation for $SU^q_2$ was done by Masuda, Mimachi, Nakagami, Noumi and Ueno \cite{mm+}, Koornwinder \cite{koo}, and Koelink and Verding \cite{kve}, by using an explicit representation found by Woronowicz in \cite{wo1}.
The problem of computing the Haar measure for Wang's free quantum groups was solved only recently, in \cite{bc1}, \cite{bc2}, \cite{bc3}. The general formula involves a 
free analogue of the Weingarten function \cite{wei}, summed over certain noncrossing partitions.

The formulae obtained so far theoretically allow to fully compute the Haar state. However, they are mathematically effective only for computing small order estimates of the moments. In fact, we have not been able so far to use them in order to derive fine analytic properties of the quantum group for fixed $n$, such as the spectral radius of the standard generators, or the existence of atoms. These questions are of potential importance in connection with the Atiyah conjecture for quantum groups, or towards the computation of the corresponding $L^2$-Betti numbers. For results in this direction, see \cite{kye} 
and \cite{cht}.

In this paper we investigate such questions for Wang's very first example of a free quantum group, namely the free analogue of $O_n$. This is an abstract object, whose  ``algebra of functions'' is defined with generators and relations, as follows:
\[A_o(n)=C^*\left((u_{ij})_{i,j=1,\ldots,n}\Big|\,u={\rm orthogonal}\right)\]

This algebra is by definition a free version of $C(O_n)$, and this gives rise to a series of interesting questions. For instance the distribution of the basic coordinates $u_{ij}\in C(O_n)$ being the hyperspherical law, the free coordinates $u_{ij}\in A_o(n)$ can be regarded as ``free hyperspherical variables'', and it is of particular interest to compute their law.

In this paper we solve the latter question. We shall prove that the density of $\sqrt{n+2}\,u_{ij}$, when pulled back on the unit circle via $z\to z+z^{-1}$, is given by:
\[F(z)=\sum_{r=-\infty}^\infty (-1)^r\frac{q^{r-1}(1+q)^2}{(1+q^{r+1})(1+q^{r-1})}z^{2r} ,\]
where $q\in (-1,0)$ is given by $q+q^{-1}=-n$ with $n> 2$. The proof uses a new method, based on the fact that certain variables in $A_o(n)$ can be modelled by variables over $SU^q_2$.
In terms of moments of the normalized variable $w=\sqrt{n+2}\,u_{ij}$, the odd ones are all zero, and the even ones are given by the following formula:
\[\int w^{2k}=\frac{q+1}{q-1}\cdot\frac{1}{k+1}\sum_{r=-k-1}^{k+1}(-1)^r\begin{pmatrix}2k+2\cr k+1+r\end{pmatrix}\frac{r}{1+q^r}\]

The possible consequences of these formulae, in connection with the considerations in \cite{bc1}, \cite{bpa}, \cite{dgg}, will be discussed in the body of the paper, and in the final section.

The paper is organized as follows. In sections 1--2 we recall the construction and basic combinatorial properties of the Wang algebra $A_o(n)$. In sections 3--5 we state and prove the main results. The final sections, 6--7, contain a few concluding remarks.

\subsection*{Acknowledgments} 

This work was started on the occasion of the Berkeley 2007 workshop ``Free probability and the large $N$ limit'', and we are highly grateful to A.~Guionnet, D.~Shlyakhtenko and D.~Voiculescu for the invitation. T.B.\ and B.C.\ were supported by the ANR ``GALOISINT''. B.C had partial support from NSERC and JSPS. P.Z.-J.\ was supported by EU Marie Curie Research Training Networks ``ENRAGE'' MRTN-CT-2004-005616, ``ENIGMA'' MRT-CT-2004-5652, ESF program ``MISGAM'' and ANR program ``GIMP'' ANR-05-BLAN-0029-01. The three authors were supported by the ANR program ``GRANMA''.

\section{Quantum groups}

Let $O_n$ be the orthogonal group. We denote by $C(O_n)$ the algebra of continuous functions $f:O_n\to\mathbb C$. This is a $C^*$-algebra, with involution and norm given by:
\begin{align*}
f^*&=(x\to\overline{f(x)})\\
||f||&=\sup|f(x)|
\end{align*}
The group operations of $O_n$, namely the multiplication, unit and inversion, produce by transposition a comultiplication, counit and antipode map for $C(O_n)$:
\begin{align*}
\Delta(f)&=((x,y)\to f(xy))\\
\varepsilon(f)&=f(1)\\
S(f)&=(x\to f(x^{-1}))
\end{align*}
Summarizing, the precise structure of $C(O_n)$ is that of a commutative Hopf $C^*$-algebra, in the sense of Woronowicz's fundamental paper \cite{wo2}.

The algebra $C(O_n)$ can be fully understood in terms of the $n^2$ coordinate functions $x_{ij}:O_n\to\mathbb C$. Indeed, it follows from the Stone--Weierstrass theorem that these functions generate $C(O_n)$, and in fact, we have the following presentation result.

\begin{theorem}
$C(O_n)$ is isomorphic to the universal $C^*$-algebra generated by $n^2$ self-adjoint commuting variables $x_{ij}$, with the relations $xx^t=x^tx=1$, where $x=(x_{ij})$ and $x^t=(x_{ji})$. The formulae
\begin{align*} 
\Delta(x_{ij})&=\sum x_{ik}\otimes x_{kj}\\
\varepsilon(x_{ij})&=\delta_{ij}\\
S(x_{ij})&=x_{ji}
\end{align*}
define morphisms of algebras, which are the comultiplication, counit and antipode.
\end{theorem}

\begin{proof}
As already mentioned, $C(O_n)$ is indeed generated by the coordinate functions $x_{ij}$, which moreover satisfy the relations in the statement. Thus $C(O_n)$ is a quotient of the universal algebra in the statement. The fact that we have indeed an isomorphism is well-known, and follows from the Gelfand theorem.
\end{proof}

The free analogue of $C(O_n)$, constructed by Wang in \cite{wa1}, is obtained by ``removing the commutativity'' from the above abstract characterization of $C(O_n)$.

\begin{definition}
$A_o(n)$ is the universal $C^*$-algebra generated by $n^2$ self-adjoint variables $u_{ij}$, with the relations $uu^t=u^tu=1$, where $u=(u_{ij})$ and $u^t=(u_{ji})$. The formulae
\begin{align*} 
\Delta(u_{ij})&=\sum u_{ik}\otimes u_{kj}\\
\varepsilon(u_{ij})&=\delta_{ij}\\
S(u_{ij})&=u_{ji}
\end{align*}
define morphisms of algebras from $A_o(n)$ to $A_o(n)\otimes A_o(n)$
(resp. $\mathbb{C}$, $A_o(n)^{op}$), called comultiplication (resp. counit and antipode).
\end{definition}

This algebra satisfies Woronowicz's axioms in \cite{wo2}. Together with the above considerations regarding $O_n$, this explains the heuristic formula $A_o(n)=C(O_n^+)$, where $O_n^+$ is a compact quantum group, called ``free version'' of $O_n$.

The study of the fine analytic properties of $A_o(n)$ will be done by using a multi-parameter deformation of this algebra, constructed by Van Daele and Wang in \cite{vwa}.
This deformation, denoted $A_o(F)$, depends on a matrix of parameters $F$. 

\begin{definition}
In what follows, $F\in GL_n(\mathbb R)$ will be a matrix satisfying $F^2=1$.
\end{definition}

We should mention that the original definition in \cite{vwa} involves an arbitrary matrix $F\in GL_n(\mathbb C)$. However, due to the various results in \cite{ban}, \cite{bdv}, not to be explained here, it is known that the relevant case is that of the matrices $F\in GL_n(\mathbb R)$ satisfying $F^2=\pm 1$. The case $F^2=1$, including the algebras $A_o(n)$, is the one that we are interested in. 

\begin{definition}
$A_o(F)$ is the universal $C^*$-algebra generated by $n^2$ elements $u_{ij}$, with the relations $u=F\bar{u}F=$ unitary, where $u=(u_{ij})$ and $\bar{u}=(u_{ij}^*)$. The formulae
\begin{align*} 
\Delta(u_{ij})&=\sum u_{ik}\otimes u_{kj}\\
\varepsilon(u_{ij})&=\delta_{ij}\\
S(u_{ij})&=u_{ji}^*
\end{align*}
define morphisms of algebras from $A_o(F)$ to $A_o(F)\otimes A_o(F)$
(resp. $\mathbb{C}$, $A_o(F)^{op}$), called comultiplication (resp. counit and antipode).
\end{definition}

This algebra satisfies also Woronowicz's axioms in \cite{wo2}. In general, the square of the antipode is no longer the identity, $S^2\neq id$.

For $F=I_n$, we have $A_o(F)=A_o(n)$. In general, $A_o(F)$ should be thought of as being a deformation of $A_o(n)$, in the same way as $C(SU^q_2)$ is a deformation of $C(SU_2)$.
For the purposes of this paper, the motivation for introducing the algebra $A_o(F)$ comes precisely from its close relation with $C(SU^q_2)$, which is as follows. 

\begin{theorem}
For any $q\in [-1,0)$, the matrix
\[F=\begin{pmatrix}0&\sqrt{-q}\cr 1/\sqrt{-q}&0\end{pmatrix}\]
satisfies $F^2=1$, and we have $A_o(F)=C(SU^q_2)$.
\end{theorem}

\begin{proof}
This is known since \cite{ban}, but since the above statement is a bit different from the original one, we will present below a self-contained proof.

First, the matrix in the statement satisfies indeed $F^2=1$. In order to compute the corresponding algebra, consider the matrix $G=\sqrt{-q}F$. We have:
\[G=\begin{pmatrix}0&-q\cr 1&0\end{pmatrix}\]
In terms of this matrix, the defining relations for $A_o(F)$ are $u=G\bar{u}G^{-1}=$ unitary. In order to find a simple writing for these relations, the first observation is that $u=G\bar{u}G^{-1}$ tells us that $u$ must be of the following special form:
\[u=\begin{pmatrix}\alpha&-q\gamma^*\cr \gamma&\alpha^*\end{pmatrix}\]
In other words, $A_o(F)$ is the universal algebra generated by two elements $\alpha,\gamma$, with the relations making $u$ unitary. But these unitarity conditions are:
\begin{enumerate}
\item $\alpha\gamma=q\gamma\alpha$.
\item $\alpha\gamma^*=q\gamma^*\alpha$.
\item $\gamma\gamma^*=\gamma^*\gamma$.
\item $\alpha^*\alpha+\gamma^*\gamma=1$.
\item $\alpha\alpha^*+q^2\gamma\gamma^*=1$.
\end{enumerate}
We recognize here the relations in \cite{wo2} defining the algebra $C(SU^q_2)$, and it follows that we have an isomorphism of algebras $A_o(F)\simeq C(SU^q_2)$ as in the statement. 
Since the Hopf algebra structure of $C(SU^q_2)$ given in \cite{wo1} is precisely the one making $u$ a corepresentation, this isomorphism is a Hopf algebra one, and we are done.
\end{proof}

\section{Diagrams and integrals}

In this section we discuss a number of basic representation theory results regarding $A_o(F)$, and  we start the computation of the laws of coefficients.

It is known since Woronowicz's paper \cite{wo1} that the representation theory of $SU^q_2$ is quite similar to that of $SU_2$, in the sense that the ``diagrams are the same''. This was extended in \cite{ban} to the case of quantum groups associated to arbitrary algebras of type $A_o(F)$. In what follows we present a brief description of this material.
We first recall that, according to the general theory in \cite{wo2}, $u$ is a corepresentation of $A_o(F)$. Its tensor powers are by definition the following corepresentations: 
\[(u^{\otimes k})_{i_1\ldots i_k,j_1\ldots j_k}=u_{i_1j_1}\ldots
u_{i_kj_k}\]
The basic representation theory problem is to compute the fixed vectors of these tensor powers,
which we denote by $Fix(u^{\otimes k})$.
First, we have here the following elementary result.

\begin{proposition}
\label{prop-TK}
Let $u$ be the fundamental corepresentation of $A_o(F)$.
\begin{enumerate}
\item $Fix(u^{\otimes 2k+1})=\{0\}$.

\item $Fix(u^{\otimes 2k})\neq\{0\}$.
\end{enumerate}
\end{proposition}

\begin{proof}
(1) It follows from definitions that $u_{ij}\to -u_{ij}$ defines an automorphism of $A_o(F)$, and this prevents $u^{\otimes 2k+1}$ from having non-trivial fixed vectors.

(2) We first examine the case $k=1$. We have:
\begin{align*}
u^{\otimes 2}
&=u\otimes u\\
&=u\otimes F\bar{u}F\\
&=(1\otimes F)(u\otimes\bar{u})(1\otimes F)
\end{align*}
Now since $u$ is unitary, it follows from definitions that $u\otimes\bar{u}$ fixes the canonical vector of $\mathbb C^n\otimes\mathbb C^n$, namely $\xi=\Sigma e_i\otimes e_i$. Thus $u^{\otimes 2}$ fixes the following vector:
\begin{align*}
\xi_F
&=(1\otimes F)\xi\\
&=\sum e_i\otimes Fe_i\\
&=\sum F_{ij}e_i\otimes e_j
\end{align*}
In the general case now, $u^{\otimes 2k}$ fixes the vector $\xi_F^{\otimes k}$, and we are done.
\end{proof}

In order to explicitly describe the space of fixed vectors of $u^{\otimes 2k}$, we introduce the following definition.

\begin{definition}
$TL(k)$ is the set of Temperley--Lieb diagrams,
\[p=\left\{ \begin{matrix}\cap\Cap\cap & \leftarrow &k\mbox{ strings}\cr \cdot\,\cdot\,\cdot\,\cdot\,\cdot& \leftarrow &2k\,\,\mbox{points}\end{matrix}\right\}\]
where the strings join pairs of points, and do not cross.
\end{definition}
The strings, as in fact the whole pictures themselves, are taken up to planar isotopy. We consider the strings as being oriented, going from left to right.
Observe that the elements of $TL(k)$ are in one-to-one correspondence with the noncrossing pair-partitions of the set $\{1,\ldots,2k\}$. With this identification, a string $\{s_l\curvearrowright s_r\}$ of the diagram corresponds to a block $\{s_l<s_r\}$ of the partition.

\begin{definition}
Given $p\in TL(k)$ and a multi-index $i=(i_1,\ldots,i_{2k})$, we set
\[\delta_p^F(i)=\prod_{s\in p}F_{i_{s_l}i_{s_r}}\]
where the product is over all the strings $s=\{s_l\curvearrowright s_r\}$ of $p$.
\end{definition}

Here are a few examples, of certain interest for the considerations to follow:
\begin{align*}
\delta_{\cap}^F(ab)&=F_{ab}\\
\delta_{\cap\cap}^F(abcd)&=F_{ab}F_{cd}\\
\delta_{\Cap}^F(abcd)&=F_{ad}F_{bc}\\
\delta_{\cap\Cap}^F(abcdef)&=F_{ab}F_{cf}F_{de}\\
\delta_{\Cap\cap}^F(abcdef)&=F_{ad}F_{bc}F_{ef}
\end{align*}
We are now in position of stating a key technical result regarding $A_o(F)$, which will allow us to compute the Haar functional. Let $e_1,\ldots,e_n$ be the standard basis of $\mathbb C^n$.

\begin{theorem}
\label{thm-dual-TK}
The family of vectors
\[\xi_p=\sum_i\delta_p^F(i)e_{i_1}\otimes\ldots\otimes e_{i_{2k}}\]
with $p\in TL(k)$ 
spans
the space of fixed vectors of $u^{\otimes 2k}$.
\end{theorem}

\begin{proof}
This is known since \cite{ban}, the idea being the following. First, the vector $\xi_\cap$ is nothing but the vector $\xi_F$ appearing in the proof of Proposition \ref{prop-TK}.

Going back to the proof of Proposition 
\ref{prop-TK},
we see that having this vector $\xi_\cap$ fixed by $u^{\otimes 2}$ is equivalent to assuming that $u=F\bar{u}F$ is unitary.  
More precisely, since $A_o(F)$ is presented by these relations, it follows from \cite{wo3} that its category of corepresentations must be generated by $\xi_\cap$. This gives the result.
\end{proof}
Let us recall that, according to the general results of Woronowicz in \cite{wo2}, the algebra $A_o(n)$ has a Haar functional, which we denote by an integral sign.
The basic integration problem is the computation of the following integrals:
\[P_{ij}=\int u_{i_1j_1}\ldots u_{i_kj_k}\]
As explained in \cite{bc1}, it follows from the general results in \cite{wo2} that the $n^k\times n^k$ matrix $P$ formed by these integrals is the orthogonal projection onto $Fix(u^{\otimes k})$.

In the case where $k$ is odd, we can combine this observation with the first assertion in Proposition 
\ref{prop-TK},
and we get $P_{ij}=0$.
In the case where $k$ is even, we can compute $P$ by using the basis of fixed vectors provided by Theorem 
\ref{thm-dual-TK},
and we get the following result.

\begin{theorem}
\label{thm-weingarten}
We have the Weingarten type formula
\[\int u_{i_1j_1}\ldots u_{i_{2k}j_{2k}}=\sum_{pq}\delta_p^F(i)\delta_q^F(j)W_{kN}(p,q)\]
where $W_{kN}=G_{kN}^{-1}$, with $G_{kN}(p,q)=N^{l(p,q)}$,
$l(p,q)$ are the number of loops obtained by
the gluing of $p$ and $q$, and $N=\sum F_{ij}^2$.
\end{theorem}

\begin{proof}
According to the above considerations, and to some linear algebra computations explained as well in \cite{bc1}, the integral on the left can be expressed in terms of the Gram matrix of the basis in Theorem 
\ref{thm-dual-TK}.
This Gram matrix is given by:
\begin{align*}
<\xi_p,\xi_q>
&=\sum_i\delta_p^F(i)\delta_q^F(i)\\
&=N^{l(p,q)}\\
&=G_{kN}(p,q)
\end{align*}
This gives the formula in the statement.
\end{proof}

The general formula in Theorem \ref{thm-weingarten}
has a number of standard applications, most of which were already worked out in \cite{bc1}, in the $F=I_n$ case. Without getting into details, let us just mention that the main results in \cite{bc1} indeed extend.

In the rest of this paper we shall focus on the problem mentioned in the introduction, namely the computation of the law of $u_{ij}$. Our first result in this sense is the following.

\begin{proposition}
\label{prop-equivalence}
Let $n\in\mathbb N$, $n\ge 2$, and consider the number $q\in [-1,0)$ satisfying $q+q^{-1}=-n$. Then the following variables have the same law:
\begin{enumerate}
\item $\sqrt{n+2}\,u_{ij}\in A_o(n)$.
\item $\alpha+\alpha^*+\gamma-q\gamma^*\in C(SU^q_2)$.
\end{enumerate}
\end{proposition}

\begin{proof}
We first compute the moments of the variable $u_{ij}\in A_o(n)$. The odd moments are all zero, and the even ones can be computed as follows:
\begin{align*}
\int u_{ij}^{2k}
&=\int u_{ij}\ldots u_{ij}\\
&=\sum_{pq}\delta_p^{I_n}(i\ldots i)\delta_q^{I_n}(j\ldots j)W_{kn}(p,q)\\
&=\sum_{pq}W_{kn}(p,q)
\end{align*}
We denote by $v$ the fundamental corepresentation of $C(SU^q_2)$. According to the various identifications in the previous section, the variable in (2) is given by:
\begin{align*}
w
&=\alpha+\alpha^*+\gamma-q\gamma^*\\
&=v_{11}+v_{22}+v_{12}+v_{21}\\
&=\sum_{ij}v_{ij}
\end{align*}
The moments of $w$ can be computed as well by using the Weingarten formula. The odd moments are all zero, and the even ones are given by:
\begin{align*}
\int w^{2k}
&=\int\left(\sum_{ij}v_{ij}\right)^{2k}\\
&=\sum_{ij}\int v_{i_1j_1}\ldots v_{i_{2k}j_{2k}}\\
&=\sum_{ijpq}\delta_p^F(i)\delta_q^F(j)W_{kN}(p,q)
\end{align*}
Here the use the identification $C(SU^q_2)=A_o(F)$ given by Theorem 1.5, where $F$ is the following matrix:
\[F=\begin{pmatrix}0&\sqrt{-q}\cr 1/\sqrt{-q}&0\end{pmatrix}\]
Due to the special form of this matrix, the number $N$ is given by:
\begin{align*}
N
&=\sum F_{ij}^2\\
&=-q-q^{-1}\\
&=n
\end{align*}
Also, for any diagram $p$ we have:
\begin{align*}
\sum_i\delta_{pi}
&=\left(\sum_{ij}F_{ij}\right)^k\\
&=(\sqrt{-q}+1/\sqrt{-q})^k\\
&=(-q-q^{-1}+2)^{k/2}\\
&=(n+2)^{k/2}
\end{align*}
By combining together all these observations, we get:
\[\int w^{2k}=(n+2)^k\sum_{pq}W_{kn}(p,q)\]
Thus $w/\sqrt{n+2}$ has the same moments as $u_{ij}$, which gives the result.
\end{proof}

\section{Matrix formulation}

In this section and unless the contrary is explicitly stated, we shall assume in the remainder of the
paper that $n$ is an integer satisfying $n > 2$.
We have seen in the previous section that the problem of computing the law of the standard generators $u_{ij}\in A_o(n)$ reduces to a certain calculation in the quantum group $SU^q_2$, where the deformation parameter $q\in (-1,0)$ is given by $q+q^{-1}=-n$.

In this section we further improve this result, by presenting a matrix formulation of the problem. For this purpose, we use a representation found by Woronowicz in \cite{wo1}.
We denote by $\mathbb T$ the unit circle in the complex plane.

\begin{theorem}
 For any $u\in\mathbb T$, we have a representation of $C(SU^q_2)$ given by 
\begin{align*}
\pi_u(\alpha)e_k&=\sqrt{1-q^{2k}}e_{k-1}\\
\pi_u(\gamma)e_k&=uq^k e_k
\end{align*}
where $(e_k)$ is the standard basis of $\ell^2(\mathbb N)$.
\end{theorem}

\begin{proof}
This follows from the fact that the operators $\pi_u(\alpha)$ and $\pi_u(\gamma)$ satisfy the defining relations for $C(SU^q_2)$. See Woronowicz \cite{wo1}.
\end{proof}

One of the remarkable features of Woronowicz's representation is that it allows the computation of the Haar measure of $SU^q_2$, via a particularly simple formula.

\begin{theorem}
\label{thm-woro-haar}
The Haar functional of $C(SU^q_2)$ is given by
\[\int a=(1-q^2)\int_{\mathbb T}tr(D\pi_u(a))\frac{du}{2\pi iu}\]
where $D$ is the diagonal operator given by $D(e_k)=q^{2k}e_k$.
\end{theorem}

\begin{proof}
This is a result of the abstract characterization of the Haar functional, as being the unique bi-invariant state. See Woronowicz \cite{wo1}.
\end{proof}

We are now in position of continuing the main computation. We begin with an elementary result, which slightly improves the model found in Proposition \ref{prop-equivalence}. 

\begin{proposition}
\label{prop-ABconstant}
The following variables over $SU^q_2$ have the same moments:
\begin{enumerate}
\item $\alpha+\alpha^*+\gamma-q\gamma^*$.

\item $\alpha+\alpha^*+A\gamma+B\gamma^*$, for any $A,B\in\mathbb C$ with 
$AB=-q$.
\end{enumerate}
\end{proposition}

\begin{proof}
It is enough to check that the moments of the variable in (2) don't depend on the choice of $A,B$. But this follows from Theorem \ref{thm-weingarten}.
Alternatively, we can view it as a consequence of
 Theorem \ref{thm-woro-haar}, because when one expands
$(\alpha+\alpha^*+A\gamma+B\gamma^*)^k$, monomials which do not have the same degree in
$\gamma$ and $\gamma^*$ are off-diagonal and do not contribute to the Haar measure. 
Therefore the coefficient in front of any contributing monomial is a power of $AB$. This concludes the proof.
\end{proof}

Observe that when $A$ and $B$ are conjugate, the variable in (2) is self-adjoint, which makes it a natural candidate for replacing the variable in (1) from Proposition \ref{prop-equivalence}. 

We know from Theorem \ref{thm-woro-haar}
how to compute laws of variables over $SU^q_2$. Together with a suitable change of basis, this leads to the following result.

\begin{proposition}
\label{prop-law-w}
The law of $w=\alpha+\alpha^*+A\gamma+B\gamma^*$ is given by
\[\int\phi(w)=(1-q^2)\int_{\mathbb T}tr(D\phi(M))\frac{du}{2\pi iu}\]
where $D(e_k)=q^{2k}e_k$ and $M(e_k)=e_{k+1}+q^k(Au+Bu^{-1})e_k+(1-q^{2k})e_{k-1}$.
\end{proposition}

\begin{proof}
We first use the notations from the beginning of this section. 
The images of $\alpha^*,\gamma^*$ by the Woronowicz representation are given by:
\begin{align*}
\pi_u(\alpha^*)e_k&=\sqrt{1-q^{2k+2}}e_{k+1}\\
\pi_u(\gamma^*)e_k&=u^{-1}q^k e_k
\end{align*}
Together with the formulae in Theorem 3.1, this shows that the image $M=\pi_u(w)$ of the variable in the statement is given by:
\[
M(e_k)=\sqrt{1-q^{2k+2}}e_{k+1}+q^k(Au+Bu^{-1})e_k+\sqrt{1-q^{2k}}e_{k-1}
\]
On the other hand, by using Theorem \ref{thm-woro-haar},
we have:
\[\int\phi(w)=(1-q^2)\int_{\mathbb T}tr(D\phi(M))\frac{du}{2\pi iu}\]
We now perform a change of basis. Consider another copy of $\ell^2(\mathbb N)$, with orthonormal basis denoted $\{e_n'\}$, and define a linear map $J:\ell^2(\mathbb N)\to \ell^2(\mathbb N)$ as follows:
\[J(e_k')=\sqrt{1-q^2}\sqrt{1-q^4}\ldots\sqrt{1-q^{2k}}e_k\]
In terms of $D'=J^{-1}DJ$ and $M'=J^{-1}MJ$, the above formula becomes:
\[\int\phi(w)=(1-q^2)\int_{\mathbb T}tr(D'\phi(M'))\frac{du}{2\pi iu}\]
Since $J$ is diagonal, we have $D'(e_k')=q^{2k}e_k'$. As for $M'$, it is given by:
\begin{align*}
M'(e_k)
&=J^{-1}MJ(e_k')\\
&=\sqrt{1-q^2}\sqrt{1-q^4}\ldots\sqrt{1-q^{2k}}J^{-1}M(e_k)\\
&=e_{k+1}'+q^k(Au+Bu^{-1})e_k'+(1-q^{2k})e_{k-1}'
\end{align*}
In other words, when removing all the prime signs for the above integration formula, we obtain the formula in the statement.
\end{proof}

\section{Orthogonal polynomials}

In this section we state and prove the main result in this paper. We will compute the density of the variable in Proposition \ref{prop-law-w}, 
by using orthogonal polynomials, and a method inspired from the paper of Koeling and Verding \cite{kve}.
In order to state the main result, we introduce the following notation.

\begin{definition}
A variable $w$ is said to have circular density $F$ if
\[\int\phi(w)=\int_{\mathbb T}F(z)\phi(z+z^{-1})\frac{dz}{2\pi iz}\]
for any continuous function $\phi$, and $F(z)=F(z^{-1})$ for all $z$.
\end{definition}

With the change of variables $z=e^{it}$, this formula becomes:
\[\int\phi(w)=\frac{1}{\pi}\int_0^{\pi}F(e^{it})\phi(2\cos t)\,dt\]

In particular, if $w$ is self-adjoint, its spectrum must be contained in $[-2,2]$. 

\begin{theorem}
\label{thm-main-computation}
The circular density of $w=\sqrt{n+2}\,u_{ij}$ is given by
\begin{align*}
F(z)&=\frac{1}{2}(G(z)+G(z^{-1}))\\
G(z)&=(1-q^2)(1-z^4)\sum_{k=0}^\infty
\frac{q^{2k}(1-q^{2k+1}z^2)}
{(1+q^{2k}z^2)(1+q^{2k+1}z^2)(1+q^{2k+2}z^2)}
\end{align*}
where $q\in(-1,0)$ is such that $q+q^{-1}=-n$, with $n > 2$.
\end{theorem}

\begin{proof}
According to the various abstract results in Proposition \ref{prop-equivalence},
Proposition 3.3 and Proposition \ref{prop-law-w},
the law of $w$ is given by the formula:
\[\int\phi(w)=(1-q^2)\int_{\mathbb T}tr(D\phi(M))\frac{du}{2\pi iu}\]
Here, $D$ is the diagonal operator on $\ell^2(\mathbb N)$ given by $D(e_k)=q^{2k}e_k$, and $M$ is the following tridiagonal operator:
\[
M(e_k)=e_{k+1}+q^k(Au+Bu^{-1})e_k+(1-q^{2k})e_{k-1}
\]
According to Proposition \ref{prop-ABconstant}, we are free to choose the constants $A$ and $B$ as long as $AB=-q$; here we assume them to be complex conjugate of each other, and therefore of modulus less than $1$. 
We compute the above integral by using a method inspired from the paper of Koelink and Verding \cite{kve}. The idea will be to express $tr(D\phi(M))$ as an integral, depending on $u$. The integrand, after being averaged over $u$, will simplify.

{\bf Step 1.} Let us first recall some basic definitions and notations from the theory of $q$-hypergeometric series, to be heavily used in this proof. These series were introduced for a number of reasons, among others in order to deal with tridiagonal operators. 

First, the $q$-shifted factorial is given by:
\[(a;q)_k=(1-a)(1-qa)\ldots(1-q^{k-1}a)\]
Its multivariable version is given by:
\[(a_1,\ldots,a_r;q)_k=(a_1;q)_k\ldots(a_r;q)_k\]
The $q$-hypergeometric series is given by:
\[{}_r\varphi_s\left(\begin{matrix}{a_1,\ldots,a_r}\cr {b_1,\ldots,b_s}\end{matrix}\,;q,z\right)
=\sum_{k=0}^\infty\left( (-1)^k q^{k(k-1)/2}\right)^{s-r+1}\frac{(a_1,\ldots,a_r;q)_k}{(q,b_1,\ldots,b_s;q)_k}\,z^k\]
The Al-Salam--Chihara polynomials, depending on two parameters $a,b$, are defined by the following formula:
\[
Q_k(x)=\frac{(ab;q)_k}{a^k}\,{}_3\varphi_2\left(\begin{matrix}q^{-k},az,az^{-1}\cr ab,0\end{matrix}\,;q,q\right)\]
Here we use the following convenient parameterization: $2x=z+z^{-1}$.
These polynomials are known to satisfy the following recurrence relation:
\[
2xQ_k(x)=Q_{k+1}(x)+q^k(a+b)Q_k(x)+(1-q^k)(1-abq^{k-1})Q_{k-1}(x)
\]
We refer to the book \cite{awi} for a complete discussion of this material.
We begin now the computation. We specify the two parameters $a$ and $b$ to be:
\begin{align*}
a&=Au\\
b&=Bu^{-1}
\end{align*}
In terms of $a,b$, the tridiagonal formula of $M$ becomes: 
\[
M(e_k)=e_{k+1}+q^k(a+b)e_k+(1-q^{2k})e_{k-1}
\]
Now recall that we have $AB=-q$, so $ab=-q$. Thus the recurrence formula for the corresponding Al-Salam--Chihara polynomials is:
\[
2xQ_k(x)=Q_{k+1}(x)+q^k(a+b)Q_k(x)+(1-q^{2k})Q_{k-1}(x)
\]
We deduce that the collection $(Q_k(x))$ of Al-Salam--Chihara polynomials evaluated at $x$ play the role of eigenvectors for the operator $M$ with corresponding eigenvalue $2x$.

The Al-Salam--Chihara polynomials are orthogonal with respect to
the Askey--Wilson measure $\mu$ with two complex conjugate parameters $a$, $b$ (and the other two set to zero), 
which in the case $|a|=|b|<1$ is a continuous measure $d\mu(x)=\beta(x)dx/\sqrt{1-x^2}$ with
\[\beta(x)=\frac{(q,ab;q)_\infty}{2\pi}\cdot 
\frac{(z^2;z^{-2};q)_\infty}{(a z,az^{-1},bz,bz^{-1};q)_\infty}\]
where once again $2x=z+z^{-1}$.

We also introduce $P_t$, which is the Poisson kernel for the Al-Salam--Chihara polynomials, i.e.\ the kernel of the operator
of multiplication by $t$ in the basis of the $e_n$:
\[P_t(x,y)=\sum_{k=0}^\infty\frac{Q_k(x)Q_k(y)}{(q,ab;q)_k}t^k\]
where the denominator is the squared norm of $Q_k$ w.r.t.\ $\mu$.

It is now a consequence of the general theory, cf. \cite{ars}, that  computing $tr(D\phi(M))$ amounts to integrating $\phi$ multiplied by the diagonal of the Poisson kernel with parameter $t=q^2$, with respect to the Askey--Wilson measure. In other words, we have:
\[tr(D\phi(M))=\int\phi(2x)P_{q^{2}}(x,x)d\mu (x)\]

{\bf Step 2.} In order to compute $tr(D\phi(M))$ by using the integral found in Step 1, the first task is to find a tractable expression for the integrand.
For this purpose, we use a remarkable formula from \cite{ars}. This formula states that we have a decomposition $P_t(x,y)=\alpha\varphi$, where:
\begin{align*}
\alpha&=\frac{(atz,atz^{-1},btw,btw^{-1},t;q)_\infty}{(tzw,tzw^{-1},tz^{-1}w,tz^{-1}w^{-1},abt;q)_\infty}\\
\varphi&={}_8 W_7\left(\frac{abt}{q};t,bz,bz^{-1},a w,a w^{-1};q,t\right)
\end{align*}
with the notations 
$${}_8 W_7(\alpha;\beta_1,\ldots,\beta_5;q,z)={}_8 \varphi_7 
\left(\begin{matrix}
\alpha,q\sqrt{\alpha},-q\sqrt{\alpha},\beta_1,\ldots,\beta_5
\cr
\sqrt{\alpha},-\sqrt{\alpha},q\alpha/\beta_1,\ldots,q\alpha/\beta_5\end{matrix}; q,z\right)$$
as well as $2x=z+z^{-1}$ and $2y=w+w^{-1}$.

In particular, with $t=q^2$, $ab=-q$ (coming from the formula $AB=-q$),
so that $\sqrt{\alpha}=i q$, 
and $2x=2y=z+z^{-1}$, 
we have $P_{q^2}(x,x)=\alpha\varphi$, where:
\begin{align*}
\alpha&=\frac{(q^2az,q^2az^{-1},q^2bz,q^2bz^{-1},q^2;q)_\infty}{(q^2z^2,q^2,q^2,q^2z^{-2},-q^3;q)_\infty}\\
\varphi&={}_8\varphi_7\left(\begin{matrix}
-q^2,-iq^2,iq^2,q^2,bz,bz^{-1},az,az^{-1}\cr
-iq,iq,-q,-q^3b^{-1}z^{-1},-q^3b^{-1}z,
-q^3a^{-1}z^{-1},-q^3a^{-1}z\end{matrix}\,;q,q^2\right)
\end{align*}
Summarizing, our integrand is given by $P_{q^2}(x,x)d\mu(x)=\alpha\varphi\beta dx/\sqrt{1-x^2}$, where $\alpha,\varphi$ are the above quantities, and $\beta$ is the density of the Askey-Wilson measure.

Let us first compute $\alpha$. We have:
\begin{align*}
\alpha
&=\frac{(q^2az,q^2az^{-1},q^2bz,q^2bz^{-1},q^2;q)_\infty}{(q^2z^2,q^2,q^2,q^2z^{-2},-q^3;q)_\infty}\\
&=\frac{(q^2;q)_\infty}{(q^2,q^2,-q^3;q)_\infty}\cdot\frac{(q^2az,q^2az^{-1},q^2bz,q^2bz^{-1};q)_\infty}{(q^2z^2,q^2z^{-2};q)_\infty}\\
&=\frac{1}{(q^2,-q^3;q)_\infty}\cdot\frac{(az,az^{-1},bz,bz^{-1};q)_\infty}{(az,az^{-1},bz,bz^{-1};q)_2}\cdot\frac{(z^2,z^{-2};q)_2}{(z^2,z^{-2};q)_\infty}\\
&=\frac{1-q^4}{(q^2;q^2)_\infty}\cdot\frac{(z^2,z^{-2};q)_2}{(az,az^{-1},bz,bz^{-1};q)_2}\cdot\frac{(az,az^{-1},bz,bz^{-1};q)_\infty}{(z^2,z^{-2};q)_\infty}
\end{align*}
By developing $\varphi$ according to the definition given in Step 2, and by using the formula $ab=-q$, we get:
\begin{align*}
\varphi
&=\sum_{k=0}^\infty
\frac{(-q^2,-iq^2,iq^2,q^2,bz,bz^{-1},az,az^{-1};q)_k}
{(q,-iq,iq,-q,-q^3b^{-1}z^{-1},-q^3b^{-1}z,
-q^3a^{-1}z^{-1},-q^3a^{-1}z;q)_k}\,q^{2k}\\
&=\sum_{k=0}^\infty
\frac{(-q^2,-iq^2,iq^2,q^2,bz,bz^{-1},az,az^{-1};q)_k}
{(q,-iq,iq,-q,q^2az^{-1},q^2az,q^2bz^{-1},q^2bz;q)_k}\,q^{2k}\\
&=\sum_{k=0}^\infty
\frac{(q^2,-q^2,iq^2,-iq^2;q)_k}{(q,-q,iq,-iq;q)_k}
\cdot\frac{(az,az^{-1},bz,bz^{-1};q)_k}{(q^2az,q^2az^{-1},q^2bz,q^2bz^{-1};q)_k}\,q^{2k}\\
&=\sum_{k=0}^\infty\frac{(q^8;q^4)_k}{(q^4;q^4)_k}\cdot\frac{(az,az^{-1},bz,bz^{-1};q)_2}{(q^kaz,q^kaz^{-1},q^kbz,q^kbz^{-1};q)_2}\,q^{2k}\\
&=\sum_{k=0}^\infty\frac{1-q^{4(k+1)}}{1-q^4}\cdot\frac{(az,az^{-1},bz,bz^{-1};q)_2}{(q^kaz,q^kaz^{-1},q^kbz,q^kbz^{-1};q)_2}\,q^{2k}
\end{align*}
Combining the expressions of $\alpha$, $\beta$ and $\varphi$, we find:
\[
\alpha\varphi\beta
=(z^2,z^{-2};q)_2\sum_{k=0}^\infty\frac{q^{2k}(1-q^{4(k+1)})}{(q^kaz,q^kaz^{-1},q^kbz,q^kbz^{-1};q)_2}
\]

{\bf Step 3.} We are now in a position to integrate. With $2x=z+z^{-1}$ we have:
\begin{align*}
tr(D\phi(M))
&=\frac{1}{2}\int_{\mathbb T}\phi(z+1/z)\alpha(z)\varphi(z)\beta(z)\frac{dz}{iz}
\end{align*}
This shows that $w$ has a circular density, given by:
\[
F(z)=\frac{1}{2}(1-q^2)(z^2,z^{-2};q)_2\sum_{k=0}^\infty q^{2k}(1-q^{4(k+1)})\int_{\mathbb T}\frac{du}{2\pi i u(q^kaz,q^kaz^{-1},q^kbz,q^kbz^{-1};q)_2}\,\]
Replacing $a=Au$ and $b=B u^{-1}$,
we compute the integral over $u$ as a sum of its four residues say inside the unit circle:
$q^k B z^{\pm 1}$, $q^{k+1} B z^{\pm 1}$.
After a tedious calculation, and using $AB=-q$, we find:
\begin{align*}
F(z)&=\frac{1}{2}(G(z)+G(z^{-1}))\\
G(z)&=(1-q^2)(1-z^4)\sum_{k=0}^\infty
\frac{q^{2k}(1-q^{2k+1}z^2)}
{(1+q^{2k}z^2)(1+q^{2k+1}z^2)(1+q^{2k+2}z^2)}
\end{align*}

This gives the formula in the statement.
\end{proof}

\section{Some consequences}

The formula of the circular density $F$ found in the previous section might look quite complicated. Here are two reformulations of Theorem
\ref{thm-main-computation}.

\begin{theorem}
\label{thm-circular-density}
The circular density of $w=\sqrt{n+2}\,u_{ij}$ is given by $F(z)=\frac{1}{2}(G(z)+G(z^{-1}))$ with
\begin{align*}
G(z)
&=\frac{1+q}{1-q}\left( 1-z^2+2(1-z^4)\sum_{k=1}^\infty (-1)^k\frac{q^k}{1+q^k z^2}\right)\\
&=1+2\sum_{r=1}^\infty(-1)^r\frac{q^{r-1}(1+q)^2}{(1+q^{r+1})(1+q^{r-1})}z^{2r}
\end{align*}
where $q\in(-1,0)$ is such that $q+q^{-1}=-n$.
\end{theorem}

\begin{proof}
The first expression is the decomposition
of $G(z)$ from Theorem 
\ref{thm-main-computation}
as a sum over its poles. 
The second is its power series expansion.
\end{proof}

This results in the computation of moments:

\begin{theorem}
\label{thm-moments}
The even moments of $w=\sqrt{n+2}\,u_{ij}$ are given by
\[\int w^{2k}=\frac{q+1}{q-1}\cdot\frac{1}{k+1}\sum_{r=-k-1}^{k+1}(-1)^r\begin{pmatrix}2k+2\cr k+1+r\end{pmatrix}
\frac{r}{1+q^r}\]
and the odd moments are all zero.
\end{theorem}

\begin{proof}
First, observe that if $F(z)=\frac{1}{2}(G(z)+G(z^{-1}))$, then for any continuous function
$\phi$, we can substitute $F$ with $G$ in the evaluation:
\[
\int \phi(w)
=\int_{\mathbb T} F(z) \phi(z+z^{-1}) \frac{dz}{2\pi i z}
=\int_{\mathbb T} G(z) \phi(z+z^{-1}) \frac{dz}{2\pi i z}
\]

Now we apply Theorem 5.1:
\begin{align*}
\int w^{2k}
&=\int_{\mathbb T}G(z)(z+z^{-1})^{2k}\frac{dz}{2\pi iz}\\
&=\int_{\mathbb T}\left(1+2\sum_{r=1}^\infty(-1)^r\frac{q^{r-1}(1+q)^2}{(1+q^{r+1})(1+q^{r-1})}z^{2r}\right)(z+z^{-1})^{2k}\frac{dz}{2\pi iz}\\
&=\begin{pmatrix}2k\cr k\end{pmatrix}+\sum_{r=1}^k(-1)^r\begin{pmatrix}2k\cr k-r\end{pmatrix}\frac{q^{r-1}(1+q)^2}{(1+q^{r+1})(1+q^{r-1})}
\end{align*}
The computation can be continued as follows:
\begin{align*}
\int w^{2k}
&=2\,\frac{1+q}{1-q}\sum_{r=0}^k(-1)^r\begin{pmatrix}2k+1\cr k-r\end{pmatrix}\frac{1-q^{2r+1}}{(1+q^r)(1+q^{r+1})}\\
&=\frac{q+1}{q-1}\left(\frac{(2k)!}{k!(k+1)!}+\frac{2}{k+1}\sum_{r=1}^{k+1}\begin{pmatrix}2k+2\cr k+1+r\end{pmatrix}\frac{r}{1+q^r}\right)\\
&=\frac{q+1}{q-1}\cdot\frac{1}{k+1}\sum_{r=-k-1}^{k+1}(-1)^r\begin{pmatrix}2k+2\cr k+1+r\end{pmatrix}\frac{r}{1+q^r}
\end{align*}
This gives the formula in the statement.
\end{proof}
Finally, let us record a statement in terms of the unnormalized variable $u_{ij}$.
\begin{theorem}
For any integer $n\geq 2$ the law of $u_{ij}$ has the following properties:
\begin{enumerate}
\item The support is $[-2/\sqrt{n+2},2/\sqrt{n+2}]$.
\item There are no atoms.
\item The density is analytic.
\item The density of $\sqrt{n+2}\,u_{ij}$ tends uniformly towards the density of the semicircle
distribution as $n\to\infty$.
\end{enumerate}
\end{theorem}

\begin{proof}
We first handle the case $n > 2$.
The three first assertions follow from Theorem 
\ref{thm-main-computation},
because the circular density $F$ appearing there is an analytic function.
The last assertion results from letting $q\to 0$ in the second equation of Theorem 5.1.

In the case $n=2$ we only have to consider the first three assertions, and the result follows 
from the fact proved in \cite{bc1} that $u_{11}^2$ has the uniform distribution
on $[0,1]$.
\end{proof}
Observe that the last point considerably strengthens a result obtained in \cite{bc1}
where it was proved that $\sqrt{n}\,u_{11}$ is asymptotically semicircular
in the sense of Voiculescu.

\section{Concluding remarks}

We have seen in this paper that the variable $u_{ij}\in A_o(n)$ can be modelled by a 
variable over $SU^q_2$, whose law can be explicitely computed.

It is natural to wonder whether similar methods can be applied to other variables in $A_o(n)$. For instance it is known since \cite{ban} that the character of the fundamental corepresentation $\chi=\Sigma u_{ii}$ is semicircular for all the algebras $A_o(F)$, so as a consequence we can state that ``the variable $\chi\in A_o(n)$ is modelled by the variable $\chi\in C(SU^q_2)$''.
We believe, however, that with suitable definitions, the variables $u_{ij}$ and $\chi$ are in fact the only ones to have models over $SU^q_2$. A concrete result in this sense can be obtained by considering linear combinations of basic generators, $x_A=Tr(uA)$. The Weingarten formula shows that the law of such a variable appears as a kind of polynomial in $A,A^t,F,F^t$, and the modelling problem ultimately reduces to that of finding $n\times n$ and $2\times 2$ matrices having the same law, known to have only 2 solutions.

Most of the methods in this paper apply as well to the algebras $A_o(F)$ with $F^2=-1$. These include the algebras $C(SU^q_2)$ with $q\in (0,1]$. However, the situation is a bit different from the one in the case $F^2=1$, because $n$ must be even, and because there is no ``canonical'' algebra, similar to $A_o(n)$, at each such $n$. A natural question here is whether $A_o(n)$ with $n$ even has a natural twisting, but we don't know the answer.

The Weingarten function for $A_o(n)$, as well as some related quantities, like the law of $u_{ij}$ computed in this paper, are closely related to the combinatorics of the meander determinants, computed in Di Francesco, Golinelli and Guitter in \cite{dgg}. It is not clear, however, how the main results in \cite{dgg} are related to the main results in this paper.

Finally, an important question is that of understanding the relationship between the law of the basic coordinates $u_{ij}\in A_o(n)$ computed in this paper, and the hyperspherical law, which appears as law of the basic coordinates $u_{ij}\in C(O_n)$. 
In the general framework of Voiculescu's free probability theory \cite{vdn}, the basic correspondence between classical and free is that provided by the Bercovici--Pata bijection \cite{bpa}. However, in the present context this bijection appears only in the limit $n\to\infty$, and the case where $n$ is fixed corresponds to the above problem.

\section{Appendix: pictures and computations}

In this appendix, we present a few figures and computations. First, we compare our original result from \cite{bc1}, namely Theorem \ref{thm-weingarten} above (applied to the specific case of the variable $u_{11}$) with the new result in this paper, namely Theorem \ref{thm-moments} above.

The purpose of this comparison is twofold. First, it serves as an algebraic verification of Theorem  \ref{thm-moments}, and secondly, it unveils the computational power of this theorem. 

We check the moments of $\sqrt{n+2}u_{11}$ for any $n>2$, up to order $7$. There is nothing to do for orders $1,3,5,7$, as the variables
are symmetric and the moments are zero.

According to Theorem \ref{thm-weingarten}, the moments of order $2,4,6$ of $\sqrt{n+2}u_{11}$ are:
$$M_2=\frac{n+2}{n}$$
$$M_4=(n+2)^2 Tr \left(
  \begin{array}{cc}
    1 & 1 \\
    1 & 1
  \end{array} \right)
\left(
  \begin{array}{cc}
    n^2 & n \\
    n & n^2
  \end{array} \right)^{-1}
  =
  \frac{2(n+2)^2}{n(n+1)}$$
$$M_6=(n+2)^3  Tr \left(
  \begin{array}{ccccc}
    1 & 1 & 1 & 1 & 1 \\
        1 & 1 & 1 & 1 & 1 \\
    1 & 1 & 1 & 1 & 1 \\
    1 & 1 & 1 & 1 & 1 \\
    1 & 1 & 1 & 1 & 1 
      \end{array} \right)
  \left(
  \begin{array}{ccccc}
    n^3 & n^2 & n^2 & n & n^2 \\
    n^2 & n^3 & n & n^2 & n \\
n^2 & n & n^3 & n^2 & n\\
n & n^2 & n^2 & n^3 & n^2\\
    n^2 & n & n & n^2 & n^3
  \end{array} \right)^{-1}
  =
  \frac {(n+2)^3(5n-7)}{ ( n^2-2)  ( n+1) n}
  $$

Now we look at the formula of Theorem \ref{thm-moments}. This gives the following formulae for the same moments:
$$M_2=\frac{(q-1)^2}{(1+q^2)}$$
$$M_4=2\frac{(q-1)^4}{(q^2-q+1)(q^2+1)}$$
$$M_6=\frac{(q-1)^6(5q^2+7q+5)}{(1+q^2)(1+q^4)(q^2-q+1)}$$
 
It is obvious that setting $q+q^{-1}=-n$ yields indeed the same formulae as those coming from Theorem \ref{thm-weingarten}. It is interesting to observe that only with Theorem 
 \ref{thm-weingarten}, it is extremely difficult to obtain moments of order, say, more than 14,
 as the formal computation of the inverse of the Gram matrix is 
  of exponential complexity, whereas Theorem \ref{thm-moments}
 reduces it to a  problem
 of linear complexity.

The relationship between Theorem \ref{thm-moments} and Theorem 
\ref{thm-weingarten} is also a source of very interesting algebraic questions. 
For instance, it is a consequence of Theorem 
\ref{thm-weingarten} that the denominator of moments of $u_{ij}$ should
divide the meander determinant introduced in \cite{dgg}. However
we are not able to prove this directly from 
Theorem \ref{thm-moments}.
 
We end the paper with a few pictures illustrating the density of the rescaled $u_{11}$ for a few
values of $n$. Observe that the last picture is a plot of the density function with a noninteger
value of $n$ decreasing towards $2$. 

The plot is coherent with the fact that for $n=2$, $u_{11}$ is symmetric
and $u_{11}^2$ has a uniform distribution on the interval $[0,1]$ as shown in \cite{bc1}.
The validity of this phenomenon, although the formulae of Theorem \ref{thm-circular-density}
and
\ref{thm-moments}
are not clearly defined at $q=-1$, follows from the fact that, according to section 2,
the $k^{\rm th}$ moment of $u_{11}$ is a rational fraction in $n$ for any $n\geq 2$, and that these rational fractions have no poles in $[2,\infty )$ (see also \cite{dgg} for a direct proof of this result).

\begin{figure}[ht!]
\begin{center}\resizebox{!}{6cm}{\includegraphics{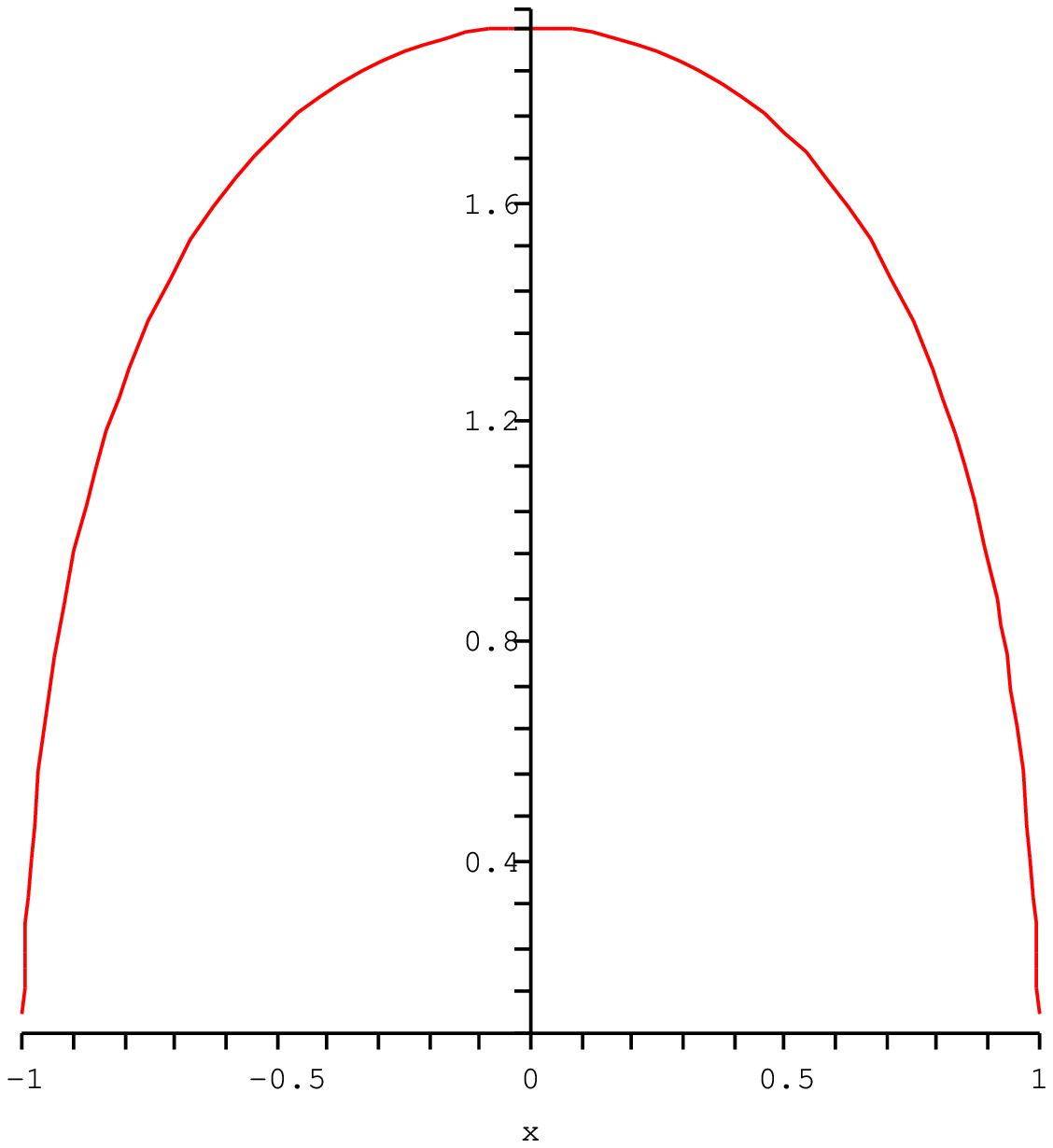}}
\end{center}
\caption{$n=50$}\label{map12}
\end{figure}

\begin{figure}[ht!]
\begin{center}\resizebox{!}{6cm}{\includegraphics{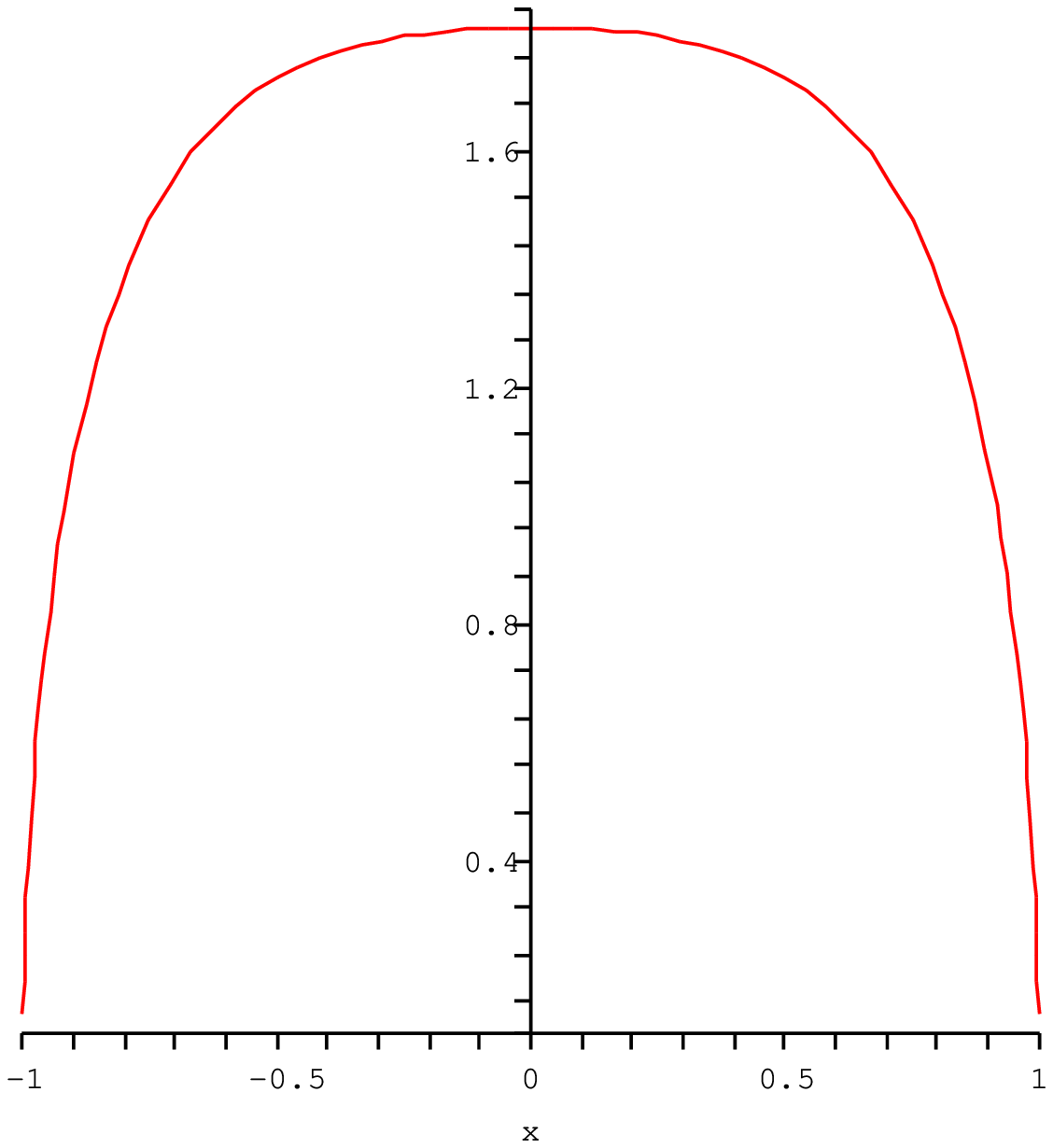}}
\end{center}
\caption{$n=20$}\label{map11}
\end{figure}

\begin{figure}[ht!]
\begin{center}\resizebox{!}{6cm}{\includegraphics{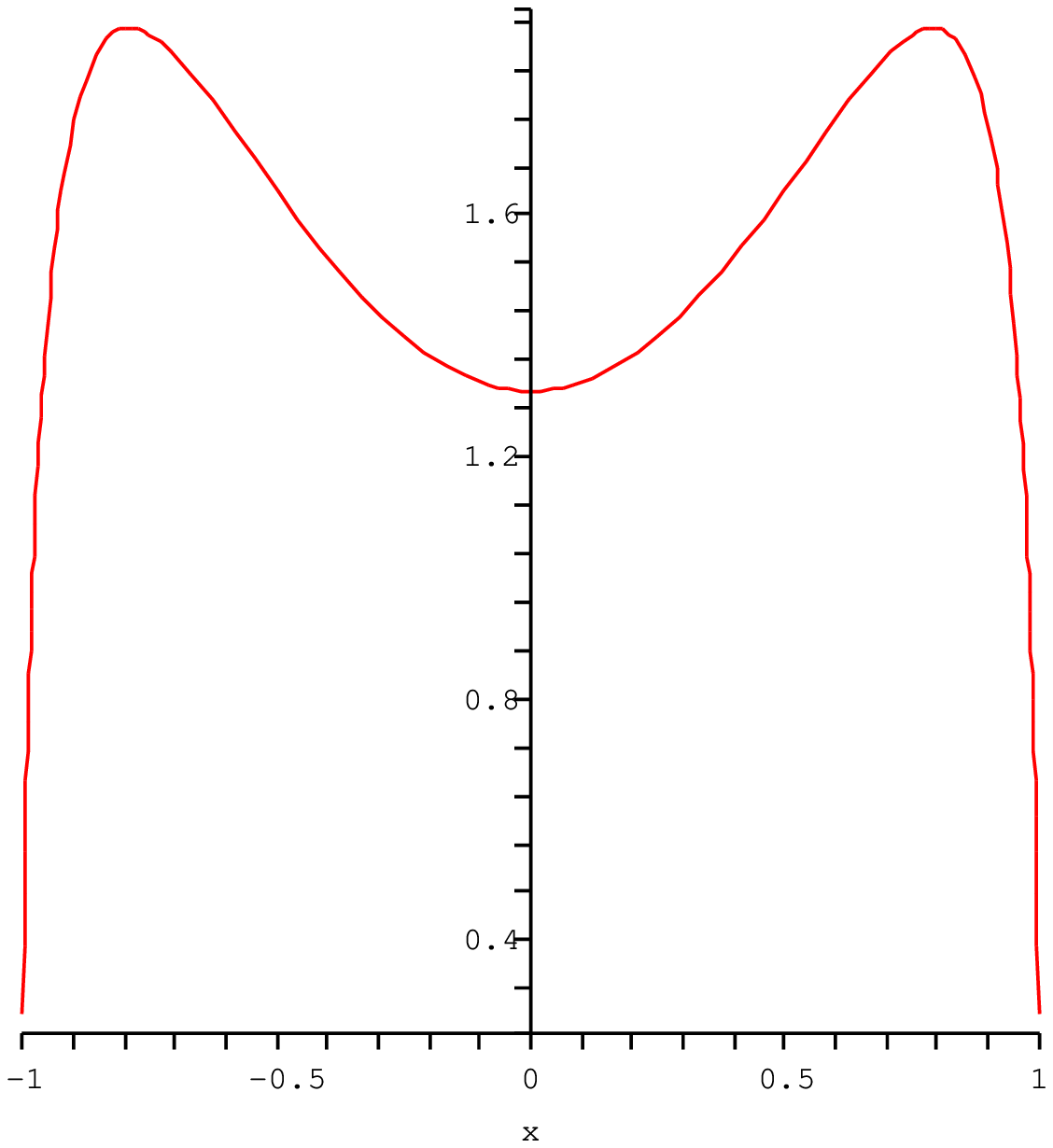}}
\end{center}
\caption{$n=5$}\label{map4}
\end{figure}

\begin{figure}[ht!]
\begin{center}\resizebox{!}{6cm}{\includegraphics{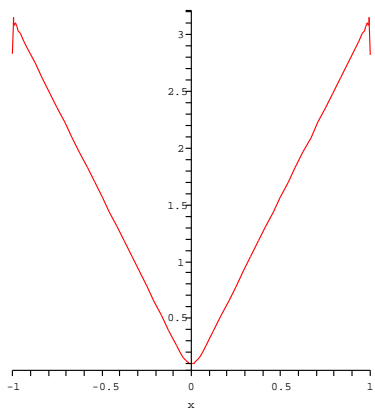}}
\end{center}
\caption{$n=2+\varepsilon$}\label{map13}
\end{figure}

\end{document}